\documentclass[12 pt]{article}
\usepackage{amsmath,amsfonts,amsthm,amssymb,mathrsfs,latexsym,paralist,xy}
\usepackage{tikz}
\usetikzlibrary{arrows,automata}
\tikzstyle{V}=[fill=black,circle,scale=0.4, outer sep = 4pt]

\usepackage{soul}

\newtheorem{thm}{Theorem}[section]
\newtheorem{prop}[thm]{Proposition}
\newtheorem{cor}[thm]{Corollary}
\newtheorem{lemma}[thm]{Lemma}
\theoremstyle{remark}
\newtheorem{rmk}[thm]{Remark}
\newtheorem{example}[thm]{Example}
\theoremstyle{definition}
\newtheorem{defn}[thm]{Definition}

\newcommand{\bi}{\begin{itemize}}
\newcommand{\ei}{\end{itemize}}
\newcommand{\be}{\begin{enumerate}}
\newcommand{\ee}{\end{enumerate}}

\renewcommand{\H}{\mathcal{H}}

\newcommand{\N}{\mathbb{N}}
\newcommand{\Z}{\mathbb{Z}}

\newcommand{\BH}{\mathcal{B}(\mathcal{H})}

\providecommand{\keywords}[1]{{\textit{Keywords and phrases:}} #1}
\providecommand{\classification}[1]{{\textit{2010 Mathematics Subject Classification:}} #1}

\newcommand{\Ran}{\operatorname{Range}}

\def\IoIIdimdots(#1/#2/#3,#4){\node at (#1,#4) {$.$};\node at (#2,#4) {$.$};\node at (#3,#4) {$.$};}
\def\IIoIIdimdots(#1,#2/#3/#4){\node at (#1,#2) {$.$};\node at (#1,#3) {$.$};\node at (#1,#4) {$.$};}

\def\IoIIIdimdots(#1/#2/#3,#4,#5){\node at (#1,#4,#5) {$.$};\node at (#2,#4,#5) {$.$};\node at (#3,#4,#5) {$.$};}
\def\IIoIIIdimdots(#1,#2/#3/#4,#5){\node at (#1,#2,#5) {$.$};\node at (#1,#3,#5) {$.$};\node at (#1,#4,#5) {$.$};}
\def\IIIoIIIdimdots(#1,#2,#3/#4/#5){\node at (#1,#2,#3) {$.$};\node at (#1,#2,#4) {$.$};\node at (#1,#2,#5) {$.$};}

\usepackage[margin=1in]{geometry}
\AtEndDocument{\bigskip{\small%
  \textsc{Carla Farsi, Judith Packer : Department of Mathematics, University of Colorado at Boulder, Boulder, CO 80309-0395, USA.} \par
  \textit{E-mail address}: \texttt{carla.farsi@colorado.edu, packer@euclid.colorado.edu} \par
  \addvspace{\medskipamount}
  \textsc{Elizabeth Gillaspy : Department of Mathematics, University of Montana, 32 Campus Drive \#0864, Missoula, MT 59812-0864, USA.}\par
  \textit{E-mail address}: \texttt{elizabeth.gillaspy@mso.umt.edu}\par
  \addvspace{\medskipamount}
  \textsc{Palle Jorgensen : Department of Mathematics, 14 MLH, University of Iowa, Iowa City, IA 52242-1419, USA.}\par
  \textit{E-mail address}: \texttt{palle-jorgensen@uiowa.edu}\par
   \addvspace{\medskipamount}
  \textsc{Sooran Kang : College of General Education, Chung-Ang University, 84 Heukseok-ro, Dongjak-gu, Seoul, Republic of Korea.} \par
  \textit{E-mail address}, \texttt{sooran09@cau.ac.kr}
}}

\begin{document}

\title{Purely atomic representations of higher-rank graph $C^*$-algebras}
\author{Carla Farsi, Elizabeth Gillaspy, Palle Jorgensen,  Sooran Kang, and Judith Packer}
\date{\today}
\maketitle

\begin{abstract}
We study purely atomic representations of $C^*$-algebras associated to row-finite and source-free higher-rank graphs. 
We describe when purely atomic representations are unitarily equivalent and we give  necessary and sufficient conditions for a purely atomic representation to be irreducible in terms of the associated projection valued measure. We also investigate the relationship between purely atomic representations, monic representations and permutative representations, and we describe when a purely atomic representation admits a decomposition consisting of permutative representations.
\end{abstract}

\classification{46L05}


\tableofcontents

\section{Introduction}

The theme of our paper falls at the crossroads of representation theory and the study of higher-rank graph $C^*$-algebras. The past two decades have seen a burst of research dealing with representations of classes of infinite $C^*$-algebras, which includes the Cuntz algebras and many Cuntz--Krieger algebras, as well as $C^*$-algebras associated to directed graphs and higher-rank graphs.

   Classical results such as J. Glimm's pioneering paper \cite{Glimm} indicate that, despite their broad applicability, the representations of purely infinite $C^*$-algebras (such as the Cuntz algebras as well as many $k$-graph $C^*$-algebras) do not admit a Borel cross-section and hence cannot be completely described.  Researchers are consequently led to study specific families of representations of these purely infinite $C^*$-algebras.  For example, the purely atomic and permutative representations of the Cuntz algebras (introduced in \cite{dutkay-jorgensen-atomic} by the third named author, together with D. Dutkay and J. Haussermann) have deep connections to  invariant subspace theory.  Kawamura and his collaborators have also identified   applications of  permutative representations of Cuntz algebras to particle physics \cite{Kawa1, Kawa2, Kawa3, KHL}.
   Purely atomic and permutative representations of Cuntz algebras also appear frequently in connection with wavelets and Walsh bases \cite{dutkay-picioroaga-silvestrov, dutkay-picioroaga, LPT}, and in applications to quantum statistical mechanics as finitely correlated states \cite{BJKW, BJ-endo, Ohno, Matsui, FNW, FNW-quantum}.  
   
    Higher-rank graphs (also called $k$-graphs) were introduced by Kumjian and Pask in \cite{KP} in order to broaden the class of $C^*$-algebras which can be studied by the combinatorial methods that had proved so fruitful in the study of Cuntz--Krieger algebras and graph $C^*$-algebras (cf.~\cite{raeburn-szyman, drinen-tomforde, bhrs,hong-syman}). 
   Like their cousins the Cuntz and Cuntz--Krieger algebras, $k$-graph $C^*$-algebras admit both a graphical and a groupoid model, as well as a description in terms of generators and relations.  In addition to their importance to $C^*$-algebraic questions, such as Elliott's classification program \cite{prrs, ruiz-sims-sorensen}, recent research has uncovered applications of $k$-graph $C^*$-algebras and their representations
    in both pure and applied mathematics, ranging from  the study of spectral triples \cite{FGJKP1, FGJKP2} and of KMS states \cite{aHKR1, aHKR2, aHLRS2, aHLRS3}, in the pure end of the spectrum,  to a long and diverse list of other applications: branching laws for endomorphisms \cite{abe-kawamura, Kawa1, FGKP, gh-Per-Frob, GLR, FGJKP-SBFS},  subshifts \cite{pask-raeburn-weaver}, endomorphisms from measurable partitions \cite{bezuglyi-jorgensen-infinite, bratteli-jorgensen-price, Bra-Jor-Ostro}, Markov measures and topological Markov chains \cite{alpay-jorgensen-Lew, bezuglyi-jorgensen, dutkay-jorgensen-monic}, wavelets and multiresolutions \cite{AJP, FGKPexcursions, MP}, signal processing and filters, iterated function systems (IFS) and fractals \cite{BJ-permutative,BJ-filters, palle-iterated, jorgensen-adv-filters, dutkay-jorgensen-fractal, dutkay-jorgensen-iterated, hutchinson}, complex projective spaces, quasi-crystals, orbit equivalence and substitution dynamical systems  and tiling systems \cite{bezuglyi-karpel, bezuglyi-k-m}.

 Motivated by   the  applications indicated in the preceding paragraphs, this paper develops the theory of purely atomic and permutative representations for the $C^*$-algebras associated to row-finite source-free higher-rank graphs.  Compared with that of the Cuntz and Cuntz--Krieger algebras, the representation theory of $k$-graph $C^*$-algebras is still in its infancy.  Representations of $C^*$-algebras associated to $k$-graphs with a single vertex are investigated systematically in \cite{davidson-yang-representations, dav-pow-yan-atomic, davidson-power-yang-dilation, Yang-End}. Indeed, the representations labeled ``atomic'' by these authors are  similar, but not precisely equivalent, to what we call  ``permutative'' representations (see Section~\ref{sec:permutative_repn}).  Similar types of representations of free semigroup algebras, which include Cuntz algebras and Toeplitz algebras, have been studied in \cite{davidson-pitts-invariant-atomic} in terms of invariant subspaces. 
More recently, four of the authors of the present paper introduced the technology of $\Lambda$-semibranching function systems for finite $k$-graphs $\Lambda$ in \cite{FGKP}, which enable the construction of representations of $C^*(\Lambda)$ on Lebesgue measure spaces $L^2(X, \mu)$.  Further examples of such representations can be found in \cite{FGKPexcursions, FGJKP-SBFS, FGJKP-monic}; see also \cite{GLR} for $\Lambda$-semibranching function systems in a broader context.

  In this paper, in addition to developing the theory of purely atomic and permutative representations of row-finite higher-rank graphs with no sources, we analyze the relationship between these representations and  the monic representations studied in \cite{FGJKP-monic}.  Although the present paper was motivated by the analysis of atomic and permutative representations of Cuntz algebras carried out in \cite{dutkay-jorgensen-atomic}, we emphasize that due to the major structural differences between Cuntz algebras and $k$-graph $C^*$-algebras, 
  the  link between our work below and \cite{dutkay-jorgensen-atomic} is more conceptual than technical. For example, the  $S_i^*$-invariant subspaces, which were a central technical tool in \cite{dutkay-jorgensen-atomic}, have no clear analogue in the $k$-graph setting.  Indeed, we have been very pleasantly surprised by the number of results from  \cite{dutkay-jorgensen-atomic} which have analogues in the present setting of higher-rank graphs, given the technical divergence between the two papers.

This paper is organized as follows. We briefly review background material including higher-rank graphs $\Lambda$, their $C^*$-algebras $C^*(\Lambda)$ and the projection valued measures associated to the representations of $C^*(\Lambda)$ in Section~\ref{sec-fund-mater}.  Any  representation $\pi$ of  $C^*(\Lambda)$ in $B(\H)$  induces a projection valued measure $P$, defined on the infinite path space $\Lambda^\infty$ of $\Lambda$ and taking values in $\text{Proj}( \H)$; see Section \ref{sec:Proj-valued-measures} for details. These projection valued measures are a central tool in our study of purely atomic representations of $C^*(\Lambda)$, which we begin in  
Section~\ref{sec:atomic_repn}.  Central results from this section include our characterization of   purely atomic representations in terms of certain subsets   $\{\text{Orbit}(\omega)\}$ of the infinite path space of $\Lambda$ which are invariant under the canonical prefixing and coding maps $\sigma_\lambda, \sigma^n$ from Definition \ref{def:infinite-path}.  These invariant subsets correspond to invariant subspaces of the Hilbert space on which $C^*(\Lambda)$ is represented.  In particular, Proposition \ref{prop-atomic-1} shows that if an irreducible representation of $C^*(\Lambda)$ admits a single atom $\omega$, the representation is purely atomic.  We subsequently show in Theorem~\ref{thm-atomic-repres} that two purely atomic representations are unitarily equivalent if and only if the corresponding projection valued measures $P$ and $\tilde{P}$ have the same support, and the ranges of $P(\omega) $ and $\tilde{P}(\omega)$, for  each atom $\omega$,  have the same dimension. Proposition~\ref{prop:irr-rank1} identifies when two purely atomic representations are disjoint and when a purely atomic representation is  irreducible.  Finally, we give  necessary and sufficient conditions for a purely atomic representation to be a monic representation in Theorem~\ref{thm:atomic-1D}.

In Section~\ref{sec:permutative_repn}, we define a permutative representation of $C^*(\Lambda)$ (Definition~\ref{defpermutative}) and investigate the conditions under which a purely atomic representation is permutative (Theorem~\ref{thm-decomp-perm-repres}).  As already mentioned, permutative representations are related to the atomic representations of \cite{dav-pow-yan-atomic}, but Proposition \ref{prop:permutative-are-sbfs} shows that they can also be realized as $\Lambda$-semibranching representations associated to a countable discrete measure space.  In Theorem~\ref{thm-decomp-perm-repres} (the main theorem of the last section), we show that a purely atomic representation is permutative if it is supported on an orbit of aperiodic paths, and we show that the representation can be decomposed into a direct sum of permutative representations with injective encoding maps.

\subsection*{Acknowledgments}   	 

 E.G.~was partially supported by   the Deutsches Forschungsgemeinschaft via the SFB 878 ``Groups, Geometry, and Actions'' of the Westf\"alische-Wilhelms-Universit\"at M\"unster. 	 
S.K.~ was supported by Basic Science Research Program through the National Research Foundation of Korea (NRF) funded by the Ministry of Education (\#2017R1D1A1B03034697). 
C.F. and J.P. were partially supported by two individual grants
from the Simons Foundation (C.F. \#523991; J.P. \#316981).

\section{Background}
\label{sec-fund-mater}

\subsection{Higher-rank graphs}

 We will now describe briefly higher-rank graphs and their $C^*$-algebras, which were introduced by Kumjian and Pask in \cite{KP}.
 
Let $\N=\{0,1,2,\dots\}$ denote the set of natural numbers and let $k\in \N$ with $k\ge 1$. We write $e_1,\dots e_k$ for the standard basis vectors of $\N^k$, where $e_i$ is the vector of $\N^k$ with $1$ in the $i$-th position and $0$ everywhere else.  We view $\N^k$ as a category with one  object (namely 0) and with composition of morphisms given by addition.

Throughout this paper, for any category $\Lambda$, the notation $\lambda \in \Lambda$ will indicate that $\lambda$ is a morphism of $\Lambda$.  Note that this is consistent with the above description of $\N^k$ as a category and the standard use of the notation $n\in \N^k$.

A countable small category $\Lambda$ with a degree functor $d:\Lambda\to \N^k$ is a \emph{higher-rank graph} or \emph{$k$-graph} if it satisfies the \emph{factorization property}: for any morphism $\lambda\in\Lambda$ and any $m, n \in \N^k$ such that  $d(\lambda)=m+n \in \N^k$,  there exist unique morphisms $\mu,\nu\in\Lambda$ such that $\lambda=\mu\nu$ and $d(\mu)=m$, $d(\nu)=n$. 

We often regard $k$-graphs as a generalization of directed graphs, so we call morphisms $\lambda\in\Lambda$ \emph{paths} in $\Lambda$, and the objects (identity morphisms) are often called \emph{vertices}. For $n\in\N^k$, we write
\begin{equation}
\label{eq:Lambda-n}
\Lambda^n:=\{\lambda\in\Lambda\,:\, d(\lambda)=n\}\
\end{equation}
With this notation, note that $\Lambda^0$ is the set of objects (vertices) of $\Lambda$. Occasionally, we call elements of $\Lambda^{e_i}$ (for any $i$) \emph{edges}.

We write $r,s:\Lambda\to \Lambda^0$ for the range and source maps in $\Lambda$ respectively.  For   vertices $v,w \in \Lambda^0$, 
\[v\Lambda w:=\{\lambda\in\Lambda\,:\, r(\lambda)=v,\;s(\lambda)=w\}.\]
  Combining this notational convention with that of Equation \eqref{eq:Lambda-n} gives, e.g.,
\[ v\Lambda^n:= \{ \lambda \in \Lambda: r(\lambda) = v, \ d(\lambda) = n\}.\]

For $m,n\in\N^k$, we write $m\vee n$ for the coordinatewise maximum of $m$ and $n$. Given  $\lambda,\eta\in \Lambda$, we write
\begin{equation}\label{eq:lambda_min}
\Lambda^{\operatorname{min}}(\lambda,\eta):=\{(\alpha,\beta)\in\Lambda\times\Lambda\,:\, \lambda\alpha=\eta\beta,\; d(\lambda\alpha)=d(\lambda)\vee d(\eta)\}.
\end{equation}
If $k=1$, then $\Lambda^{\operatorname{min}}(\lambda, \eta)$ will have at most one element; this need not be true in a $k$-graph if $k > 1$.

We say that a $k$-graph $\Lambda$ is \emph{finite} if $\Lambda^n$ is a finite set for all $n\in\N^k$ and say that $\Lambda$  \emph{has no sources} or \emph{is source-free} if $v\Lambda^n\ne \emptyset$ for all $v\in\Lambda^0$ and $n\in\N^k$. It is well known that this is equivalent to the condition that $v\Lambda^{e_i}\ne \emptyset$ for all $v\in \Lambda$ and all basis vectors $e_i$ of $\N^k$. 
We say that $\Lambda$ is \emph{row-finite} if $|v\Lambda^n| < \infty$ for all $v\in\Lambda^0$ and $n\in\N^k$. We say that a $k$-graph is \emph{strongly connected} if, for all $v,w\in\Lambda^0$, $v\Lambda w\ne \emptyset$.

To describe the infinite path space $\Lambda^\infty$ of a $k$-graph, we introduce the fundamental example $\Omega_k$ as follows.
For $k\ge 1$, let $\Omega_k$ be the small category with
\[
\operatorname{Obj}(\Omega_k)=\N^k,\quad \text{and}\quad \operatorname{Mor}(\Omega_k)=\{(p,q)\in \N^k\times \N^k\,:\, p\le q\}.
\]
Again, we can also view elements of $\text{Obj}(\Omega_k)$ as identity morphisms, via the map $\text{Obj}(\Omega_k) \ni p \mapsto (p, p) \in \text{Mor}(\Omega_k)$.
The range and source maps $r,s:\operatorname{Mor}(\Omega_k)\to \operatorname{Obj}(\Omega_k)$ are given by $r(p,q)=p$ and $s(p,q)=q$. If we define $d:\Omega_k\to \N^k$ by $d(p,q)=q-p$, then one can check that $\Omega_k$ is a $k$-graph with degree functor $d$.

\begin{defn}[\cite{KP} Definitions 2.1]
\label{def:infinite-path}
Let $\Lambda$ be a $k$-graph. An \emph{infinite path} in $\Lambda$ is a $k$-graph morphism (degree-preserving functor) $x:\Omega_k\to \Lambda$, and we write $\Lambda^\infty$ for the set of infinite paths in $\Lambda$. 
Since $\Omega_k$ has a terminal object (namely $0 \in\N^k$) but no initial object, we think of our infinite paths as having a range $r(x) : = x(0)$ but no source.
For each $m\in \N^k$, we have a shift map $\sigma^m:\Lambda^\infty \to \Lambda^\infty$ given by
\begin{equation}\label{eq:shift-map}
\sigma^m(x)(p,q)=x(p+m,q+m)
\end{equation}
for $x\in\Lambda^\infty$ and $(p,q)\in\Omega_k$. {We also have a partially defined ``prefixing map'' $\sigma_\lambda: s(\lambda) \Lambda^\infty  \to \Lambda^\infty$ for each $\lambda \in \Lambda$:
\begin{equation} \sigma_\lambda(x) = \lambda x = \left[ (p, q) \mapsto \begin{cases} \lambda(p, q), & q \leq d(\lambda) \\
x(p-d(\lambda), q-d(\lambda)), & p \geq d(\lambda) \\
\lambda (p, d(\lambda)) x(0, q-d(\lambda)), & p < d(\lambda) < q
\end{cases} \right]
\label{eq:prefixing}
\end{equation}}
Observe that, since $x: \Omega_k \to \Lambda$ is degree preserving and $\Omega_k$ contains infinitely many edges of each degree $e_i\in \N^k$, the same must be true of any infinite path $x \in \lambda^\infty$.
\end{defn}

\begin{rmk}
\label{rmk:colored-graph-perspective-paths}
We now describe an alternative perspective on higher-rank graphs,   which motivates their description as generalizations of directed graphs.  Let $G$ be an edge-colored directed graph with $k$ different colors of edges, and define an additive map $d: \text{Path}(G) \to \N^k$ by sending an edge $e$ of color $i, \ 1 \leq i \leq k$, to $d(e) = e_i \in \N^k$.  For the rest of this discussion, let ``red'' and ``blue'' denote different but arbitrary colors appearing in $G$.  

Suppose that, for every pair of vertices $v, w \in G^0$, the number of paths from $v$ to $w$ which consist of a red edge followed by a blue edge is the same as the number of paths from $v$ to $w$ consisting of a blue edge followed by a red edge.  Then we can ``pair up'' these paths and define $\Lambda_G$ to be the quotient of $G$ which arises from identifying these commuting squares.
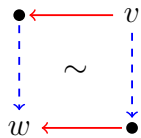
\[\begin{tikzpicture}[scale = 1.5]
			\node (W) at (0,0){$w$};
			\node (V2) [V]  at (0,1){};
			\node (V1) [V] at (1,0){};
			\node (V) at (1,1){$v$};
			\node at (0.5, 0.5) {$\sim$};
		\begin{scope}[->,  line width=0.7, color = red]
			\draw  (V1) to (W);
			\draw (V) to (V2);
		\end{scope}
		\begin{scope}[->, dashed, line width = 0.7, color = blue] 
			\draw (V) to (V1);
			\draw (V2) to (W);
		\end{scope}
	\end{tikzpicture} \]
Hazlewood et al.~proved in \cite{HRSW} that $\Lambda_G$ is in fact a $k$-graph, as long as every path in $G$ consisting of 3 edges of different colors determines a well-defined ``commuting cube'' in $\Lambda_G$.

Thus, an edge-colored directed graph with $k$ different colors of edges, and the same number of blue-red paths as red-blue paths between each pair of vertices, gives rise to a $k$-graph once we specify the pairing between the red-blue and blue-red paths.  By abuse of notation we will occasionally use the phrase ``factorization rule'' to denote this pairing.  

In this perspective, one can view each infinite path $x\in \Lambda_G^\infty$ as the equivalence class of a composable sequence of edges $x = f_1 f_2 \cdots$ in $G$, such that each of the $k$ possible  edge colors occurs infinitely often.
\end{rmk}

It is well-known that the collection of cylinder sets 
\[
Z(\lambda)=\{x\in\Lambda^\infty\,:\, x(0,d(\lambda))=\lambda\},
\]
for $\lambda \in \Lambda$, form a compact open basis for a locally compact Hausdorff topology on $\Lambda^\infty$, under  reasonable hypotheses on $\Lambda$ (in particular, when $\Lambda$ is row-finite: see Section 2 of \cite{KP}). If a $k$-graph $\Lambda$ is  finite, then $\Lambda^\infty$ is compact in this topology. 

%
%
%
%

Now we introduce the $C^*$-algebra associated to a $k$-graph $\Lambda$. Here we only consider row-finite $k$-graphs with no sources. For $C^*$-algebras associated to more general $k$-graphs, see for example  \cite{RSY2, FMY}.
\begin{defn}\label{def:kgraph-algebra}
Let $\Lambda$ be a row-finite $k$-graph with no sources. A \emph{Cuntz--Krieger $\Lambda$-family} is a collection  $\{t_\lambda:\lambda\in\Lambda\}$ of partial isometries in a $C^*$-algebra satisfying
\begin{itemize}
\item[(CK1)] $\{t_v\,:\, v\in\Lambda^0\}$ is a family of mutually orthogonal projections,
\item[(CK2)] $t_\lambda t_\eta=t_{\lambda\eta}$ if $s(\lambda)=r(\eta)$,
\item[(CK3)] $t^*_\lambda t_\lambda=t_{s(\lambda)}$ for all $\lambda\in\Lambda$,
\item[(CK4)] for all $v\in\Lambda$ and $n\in\N^k$, we have
\[
t_v=\sum_{\lambda\in v\Lambda^n} t_\lambda t^*_\lambda.
\]
\end{itemize}
The Cuntz--Krieger $C^*$-algebra $C^*(\Lambda)$ associated to $\Lambda$ is the universal $C^*$-algebra generated by a Cuntz--Krieger $\Lambda$-family.
\end{defn}
Thus, every Cuntz--Krieger $\Lambda$-family induces a representation of $C^*(\Lambda)$; we will often use the same notation, $\{ t_\lambda\}_{\lambda\in \Lambda}$, for the representation as for its Cuntz--Krieger $\Lambda$ family.

One can show that 
\[
C^*(\Lambda)=\overline{\operatorname{span}}\{s_\alpha s^*_\beta\,:\, \alpha,\beta\in\Lambda,\; s(\alpha)=s(\beta)\}.
\]
Indeed, (CK4) implies that for all $\lambda, \eta\in \Lambda$, we have
\begin{equation}\label{eq:CK4-2}
s_\lambda^* s_\eta=\sum_{(\alpha,\beta)\in \Lambda^{\operatorname{min}}(\lambda,\eta)} s_\alpha s^*_\beta. 
\end{equation}

Recall from Definition 2.7 of \cite{KP} that for any row-finite source-free $k$-graph $\Lambda$, we have a groupoid 
\[ \mathcal G_\Lambda =\{ (x, m-n, y) \in \Lambda^\infty \times \Z^k \times \Lambda^\infty: \sigma^m(x) = \sigma^n(y)\}\]
such that $C^*(\Lambda) \cong C^*(\mathcal G_\Lambda)$.  The groupoid $\mathcal G_\Lambda$ is \'etale \cite[Proposition 2.8]{KP} and amenable \cite[Theorem 5.5]{KP}. (See also \cite{FMY} for groupoid models for more general $k$-graphs and their $C^*$-algebras).
 \begin{defn} 
 \label{deforbit}
 Let $\Lambda$ be a row-finite $k$-graph with no sources, and let $\Lambda^{\infty}$ be the infinite path space of $\Lambda$, and $\mathcal{G}_\Lambda$ be the associated groupoid given above.
 For any $\omega\in \Lambda^{\infty},$ we set 
 $$\text{Orbit}(\omega)\;=\{\gamma\in\Lambda^{\infty}:(\gamma,n,\omega)\in {\mathcal G}_{\Lambda}\;\text{for some}\;n\in \mathbb Z^k\},$$
 i.e. $\gamma\in \text{Orbit}(\omega)$ if and only if there exist $m,\;\ell\in\mathbb N^k$ such that $\sigma^m(\gamma)\;=\sigma^{\ell}(\omega).$
 \end{defn}
 We note that $\text{Orbit}(\omega)$ is invariant under $\sigma^n$ (defined in Equation \eqref{eq:shift-map}) and $(\sigma^n)^{-1},$ for all $n\in \mathbb N^k.$ 
 {Note also that each orbit is a   Borel set, being a countable union of points (which are countable intersections of cylinder sets).}

\subsection{Projection valued measures}
\label{sec:Proj-valued-measures}

Inspired by Dutkay, Haussermann, and Jorgensen     \cite{dutkay-jorgensen-monic, dutkay-jorgensen-atomic}, here we review the projection valued measure associated to a  representation of $C^*(\Lambda)$, which was developed in \cite{FGJKP-monic}. 

Let $\Lambda$ be a row-finite $k$-graph with no sources.
Given a representation $\{ t_\lambda\}_{\lambda\in\Lambda }$ of a $k$-graph $C^*$-algebra $C^*(\Lambda)$ on a Hilbert space $\mathcal{H}$, we define a projection valued function $P$ on $\Lambda^\infty$ by 
 \begin{equation}\label{eq:pv-measure}
 P(Z(\lambda)) = t_\lambda t_\lambda^* \quad\text{for all $\lambda \in\Lambda$}.
 \end{equation}
 
According to Proposition~3.9 of \cite{FGJKP-monic}, the above function $P$ can be extended to a projection valued measure on the Borel $\sigma$-algebra $\mathcal{B}_o(\Lambda^\infty)$ generated by the cylinder sets.  The proof relies on the Kolmogorov extension theorem. 
 
 \begin{thm} \cite[Proposition~3.9]{FGJKP-monic}
 \label{conj-palle-proj-valued-measure-gen-case}
 Let $\Lambda$ be a row-finite $k$-graph with no sources.
 Given a representation $\{ t_\lambda\}_{\lambda\in\Lambda}$ of a $k$-graph $C^*$-algebra $C^*(\Lambda)$ on a Hilbert space $\mathcal{H}$, then the assignment 
 \[P(Z(\lambda)) = t_\lambda t_\lambda^*  \quad\text{for}\quad  \lambda \in \Lambda
 \]
 extends to  a projection valued measure on the Borel $\sigma$-algebra  $\mathcal{B}_o(\Lambda^\infty)$ generated by the cylinder sets  of the infinite path space $\Lambda^\infty$. 
 \end{thm}

We now summarize some properties of the projection valued measure $P$ on $\Lambda^\infty$ established in \cite{FGJKP-monic}.

 \begin{prop} \cite[Proposition~2.13]{FGJKP-monic}
 \label{prop-atomic-basic-equns} 
 Let $\Lambda$ be a row-finite, source-free $k$-graph, and fix a representation $\{t_\lambda:\lambda \in \Lambda\}$ of $C^*(\Lambda)$.
 Then: 
 \begin{itemize}\label{prop:pvm-properties}
 \item[(a)] For $\lambda, \eta \in\Lambda$ with $s(\lambda)=r(\eta)$, we have $t_\lambda P(Z(\eta)) t_\lambda^*=P(\sigma_\lambda(Z(\eta)))$, where $\sigma_\lambda$ is the {prefixing} 
  map on $\Lambda^\infty$ given in Equation \eqref{eq:prefixing}.
 
 \item[(b)] For any fixed $n \in \N^k$, we have
  \[
  \sum_{\lambda \in  f(\eta) \Lambda^n} t_\lambda P(\sigma_\lambda^{-1}(Z(\eta))) t_\lambda^* = P(Z(\eta)) ;
  \]
  \item[(c)]  For any $\lambda,\eta \in \Lambda$ with $r(\lambda)=r(\eta)$, we have $t_\lambda P(\sigma_\lambda^{-1}(Z(\eta))) = P(Z(\eta)) t_\lambda$;
  \item[(d)] When $\lambda\in\Lambda^n$, we have $t_\lambda P( Z(\eta)) =  P((\sigma^n)^{-1}(Z(\eta))) t_\lambda$, where $\sigma^n$ is defined in Equation \eqref{eq:shift-map}.
 \end{itemize}
 \end{prop}

\section{Purely atomic representations of $C^*(\Lambda)$}
\label{sec:atomic_repn}

In this section, we define purely atomic representations of $C^*(\Lambda)$ in terms of the projection valued measure being purely atomic (cf. Definition 4.1 of \cite{dutkay-jorgensen-atomic}).
While many of the results in this section were inspired by similar results established in \cite{dutkay-jorgensen-atomic} for Cuntz algebras, our proofs  in the setting of higher-rank graphs have required new techniques.

 \begin{defn} 
 \label{defatomic}
(c.f. Definition 4.1 of \cite{dutkay-jorgensen-atomic}.) 
Let $\Lambda$ be a row-finite $k$-graph with no sources.
A representation $\{ t_\lambda\}_{\lambda \in \Lambda}$ of  $C^*(\Lambda)$ on a Hilbert space ${\mathcal H}$ is called \emph{purely atomic}  if there exists a Borel subset $\Omega\subset\Lambda^{\infty}$ such that the projection valued measure $P$ defined on the Borel sets of $\Lambda^{\infty}$ as in Theorem~\ref{conj-palle-proj-valued-measure-gen-case} satisfies
\begin{enumerate}
 \item[(a)] $P(\Lambda^{\infty}\backslash \Omega)\;\;=\; 0_{\mathcal H},$
 \item[(b)] $P(\{\omega\})\not=0_{\mathcal H}$ for all $\omega\in \Omega$,
 \item[(c)] $\bigoplus_{\omega\in \Omega}P(\{\omega\})=\;\text{Id}_{\mathcal H},$
 \end{enumerate}
 where the sum on the left-hand side of (c) converges in the strong operator topology.
 \end{defn}
 Thus, a representation of $C^*(\Lambda)$ is purely atomic if the corresponding projective-valued measure is purely atomic on the Borel $\sigma$-algebra ${\mathcal B}_o(\Lambda^{\infty})$ of $\Lambda^\infty$.

\begin{example}
Consider the $2$-graph $\Lambda$ associated to the edge-colored directed graph
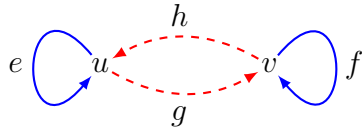
\[
\begin{tikzpicture}[scale=1.5]
 \node[inner sep=0.5pt, circle] (u) at (0,0) {$u$};
    \node[inner sep=0.5pt, circle] (v) at (1.5,0) {$v$};
    \draw[-latex, thick, blue] (v) edge [out=50, in=-50, loop, min distance=30, looseness=2.5] (v);
    \draw[-latex, thick, blue] (u) edge [out=130, in=230, loop, min distance=30, looseness=2.5] (u);
\draw[-latex, thick, red, dashed] (v) edge [out=150, in=30] (u);
\draw[-latex, thick, red, dashed] (u) edge [out=-30, in=210] (v);
\node at (-0.75, 0) {\color{black} $e$}; 
\node at (0.7, 0.45) {\color{black} $h$};
\node at (0.7, -0.45) {\color{black} $g$};
\node at (2.25, 0) {\color{black} $f$};
\end{tikzpicture}
\]
whose factorization rules are given by $eh=hf$ and $fg=ge$. With these factorization rules, there is only one infinite path $x\in u\Lambda^\infty$, namely
\[
x=ehfgehfg\dots = hfgehfge\dots = hgeehgee\dots.
\]
Similarly, one can see that there is only one infinite path $y\in v\Lambda^\infty$, namely
\[
y=fgehfgeh\dots = gehfgehf \dots = ghffghff \dots .
\]
In other words, $\Lambda^\infty \{x, y\}$.  Since $\Lambda^\infty$ is finite,   any nontrivial representation of $C^*(\Lambda)$ with the associated projection-valued measure $P$ and $\Omega$ satisfies the conditions (a), (b) and (c) of Definition~\ref{defatomic}, and hence it must be purely atomic. 
\end{example}

\subsection{Characterizations of purely atomic representations} 
 
To arrive at Theorem \ref{thm-atomic-repres}, which characterizes unitarily equivalent purely atomic representations as having projection valued measure with equal support and ranges,  we start with the following proposition. 
 
\begin{prop}
\label{prop:atomic-2}  Let $\Lambda$ be a row-finite $k$-graph with no sources, and let $\{t_\lambda\}_{\lambda \in \Lambda}$ generate a purely atomic representation of $C^*(\Lambda).$  Let $P$ be the associated projection valued measure on the Borel subsets of the infinite path space $\Lambda^{\infty}.$ Then we have the following.
\begin{itemize}
\item[(a)] For $\lambda\in\Lambda$ and $\omega\in Z(s(\lambda))\subset\; \Lambda^{\infty}$, we have 
\[
t_{\lambda}P(\{\omega\})t_{\lambda}^*\;=\;P(\{\lambda\omega\}),\]
and for $n\in \mathbb N^k$, we have 
\[
t_{\omega(0,n)}^*P(\{\omega\})t_{\omega(0,n)}\;=\;P(\{\sigma^n(\omega)\}).
\]
\item[(b)] For $\eta\in\Lambda^n$ and $\omega\in \Lambda^\infty$ with $\eta\not=\omega(0,n)$, we have
\[
t_{\eta}^*P(\{\omega\})t_{\eta}=0.
\]
\end{itemize}
\end{prop}
\begin{proof}
These statements follow from taking the limit in the strong operator topology on a nested sequence of cylinder sets decreasing to $\{\omega\}$ and using Proposition~\ref{prop-atomic-basic-equns}(a).
\end{proof}

 The following Corollary is an immediate consequence of Proposition \ref{prop:atomic-2} and our observations above.

 \begin{cor} 
 Let $\Lambda$ be a row-finite, source-free $k$-graph and let $\{ t_\lambda\}_{\lambda \in \Lambda}$ be a representation of $C^*(\Lambda)$. 
 \begin{itemize}
\item[(a)] If  $Orbit(\omega) = Orbit(\gamma)$ then $P(\{\omega\}) = 0$ iff $ P(\{\gamma\}) = 0$.
\item[(b)]  For any purely atomic representation
  $\{ t_\lambda\}_{\lambda \in \Lambda}$, define the support $\Omega = \text{supp}(P)$ of $\{ t_\lambda\}_{\lambda \in \Lambda}$ by
  \[\Omega : = \text{supp}(P)=\{ \omega \in \Lambda^\infty: P(\{\omega\}) \not= 0\}.\]
   We can decompose $\Omega$ as a disjoint union of orbits: 
  $\Omega\;=\;\bigsqcup_{\omega}\text{Orbit}(\omega).$  In particular, 
 \[\bigoplus_{\omega \in \Omega} P(\text{Orbit}(\omega))= \text{Id}_{\mathcal H}.\]
 \end{itemize}
 
\label{prop:orbits-of-atoms-are-nonzero} 
 \end{cor}

\begin{rmk}\label{rmk:inv-subsp}
It follows from Proposition \ref{prop:atomic-2} that the subspace $P(\text{Orbit}(\omega))\subseteq \H$ is   invariant for the representation $\{t_\lambda\}_{\lambda\in\Lambda}$ in the sense of \cite{dutkay-jorgensen-atomic}. 
\end{rmk}

\begin{example}
\label{exampleatomic}
(cf.~\cite[Proposition 2.11]{KP}) Recall that for a row-finite, source-free $k$-graph $\Lambda$, the infinite path representation of $C^*(\Lambda)$ first given by A. Kumjian and D. Pask via the partial isometries $\{S_{\lambda}: \lambda\in \Lambda\}$ on the non-separable Hilbert space $\ell^2(\Lambda^{\infty})$ with orthonormal basis $\{h_\omega: \omega\in \Lambda^{\infty}\}$ is given by 
$$S_{\lambda}(h_{\omega})\;=\;\delta_{s(\lambda), r(\omega)}h_{\lambda\omega},\;\;\text{and}\;\;S_{\lambda}^*h_{\omega}=\delta_{\lambda, \omega(0,d(\lambda))}h_{\sigma^{d(\lambda)}(\omega)}.$$
One can check that this representation is purely atomic.  Indeed, for all $\omega\in \Lambda^{\infty},$
$$P(\{\omega\})=\lim_{n\in \N^k} S_{ \omega(0,n)}S_{ \omega(0,n)}^*\;=\;P_{\text{span}\, h_\omega}.$$
Here the limit is taken in the strong operator topology. This is a standard example to keep in mind when considering both purely atomic representations and the permutative representations which we discuss in Section \ref{sec:permutative_repn} below.
\end{example}

We now show that any intertwiner of purely atomic representations of a $k$-graph algebra intertwines the associated projection valued measures.  

\begin{prop}
\label{propintertwine}  Let $\Lambda$ be a row-finite $k$-graph with no sources, and let  $\{ t_\lambda\}_{\lambda \in \Lambda}$ and  $\{ \tilde t_\lambda\}_{\lambda \in \Lambda}$ be two purely atomic representations of $C^{\ast}(\Lambda)$ on the Hilbert spaces ${\mathcal H}$ and ${\mathcal H}',$ respectively.  Suppose that $U:{\mathcal H}\to {\mathcal H}'$ is an intertwining operator for these representations, so that 
\[
\tilde{t_\lambda}U\;=\;U t_{\lambda}\;\;\text{and}\;\;(\tilde t_\lambda)^*U=Ut_{\lambda}^*\quad\text{for all}\;\; \lambda\in \Lambda.
\]
Let $P$ and $\tilde{P}$ be the associated projection valued measures on ${\mathcal B}(\Lambda^{\infty})$ with respect to $\{t_\lambda\}$, $\{\tilde{t}_\lambda\}$.  Then for every $\omega\in \Lambda^{\infty},$
$$\tilde{P}(\{\omega\})U\;=\;U P(\{\omega\}).$$ 
Moreover, if $U$ is a unitary operator so that the representations are unitarily equivalent, then the supports of $P$ and $\tilde{P}$ are the same.
\end{prop}
\begin{proof}
Since $U$ intertwines the representations, we see that for every $\lambda\in \Lambda,$
$$\tilde{P}(Z(\lambda))U=\tilde t_\lambda(\tilde t_\lambda)^*U\;=\;Ut_{\lambda}t_{\lambda}^*\;=\;UP(Z(\lambda)).$$
Therefore $U$ intertwines the projection valued measures on all Borel subsets in $\Lambda^{\infty},$ including the point sets. It follows that if $U$ is unitary, then for every $\omega\in \Lambda^{\infty},$ we have 
$$\tilde{P}(\{\omega\})=UP(\{\omega\})U^*,$$ so that $\text{dim}P(\{\omega\})=\text{dim}\tilde{P}(\{\omega\})$ for all $\omega\in \Lambda^{\infty},$ and  hence $\text{supp}(P)=\text{supp}(\tilde{P}).$
\end{proof}
We now derive some straightforward consequences of Corollary~\ref{prop:orbits-of-atoms-are-nonzero}.

\begin{prop}
Suppose that an irreducible representation $\{t_\lambda\}_{\lambda \in \Lambda}$  of $C^*(\Lambda)$ has an atom $\omega$.  Then $\{t_\lambda\}_{\lambda \in \Lambda}$ is purely atomic and the associated projection valued measure is supported on Orbit$(\omega)$.  
\label{prop-atomic-1}
\end{prop}
\begin{proof}
Let $\{t_\lambda\}$ be an irreducible representation of $C^*(\Lambda)$ with an atom $\omega\in \Lambda^\infty$.
 We   first  observe  that for any $x \in \Lambda^\infty$, 
 \[ P(\{x\}) t_\lambda = \begin{cases} 0, & x \not\in Z(\lambda) \\
 t_\lambda P(\{\sigma^{d(\lambda)}(x)\}), & x \in Z(\lambda).
 \end{cases}\]
  This follows from writing $P(\{x\}) = \lim \{ t_{\eta_n} t_{\eta_n}^*: x \in Z(\eta_n)\}$, with $\eta_n=x(0,n)$, and observing that if $d(\eta) \geq d(\lambda)$, then 
  \[ t_\eta^* t_\lambda = \sum_{(\rho, \nu) \in \Lambda^{\min}(\eta, \lambda)} t_\nu t_\rho^* = \begin{cases} t_\rho^*, & \eta = \lambda \rho \\ 
  0, & \text{ else.}
  \end{cases}\]
  Therefore, if $x \not\in Z(\lambda)$ then we can find $Z(\eta)$ that contains $x$ with $d(\eta) \geq d(\lambda)$ and such that $\eta$ does not extend $\lambda$.
  Consequently, if we set 
  \[ P := \sum_{x\not\in Orbit(\omega)} P(\{x\})\]
(where the limit defining the sum is taken in the strong operator topology),  
   then $ P t_\lambda   = t_\lambda \sum_{x \in Z(\lambda) \backslash Orbit(\omega)} P(\{ \sigma^{d(\lambda)}(x)\}).$
Moreover, $P$ is a projection since $P(\{x\}) P(\{y\}) = 0$ for $x\not= y$.  
  
  On the other hand, 
  \[t_\lambda P = t_\lambda \sum_{ \{y \not\in Orbit(\omega): r(y) = s(\lambda)\}} P(\{y\}).\]
  Since $\{y \not\in Orbit(\omega): r(y) = s(\lambda)\} = \{ \sigma^{d(\lambda)}(x): x \in Z(\lambda) \backslash Orbit(\omega)\}$, we have 
  \[ t_\lambda P = P t_\lambda\]
  for any $\lambda \in \Lambda$.  
  Our assumption that $\{t_\lambda\}_{\lambda}$ is irreducible now implies that $P$ must be a multiple of the identity.  However, $P < 1$ since $\omega$ is an atom, so we must have $P =0$.
(When $\Lambda^\infty$ has only one point, $P=1$ on the atom, and hence $\{t_\lambda\}_{\lambda\in\Lambda}$ is purely atomic).
Corollary~\ref{prop:orbits-of-atoms-are-nonzero} now implies that $P(\{x\}) \not= 0$ for every $x \in Orbit(\omega)$, completing the proof that $\{t_\lambda\}_{\lambda \in \Lambda}$ is purely atomic.
\end{proof}

The following result was inspired by Corollary 4.8 of \cite{dutkay-jorgensen-atomic}, but the technical details are much more intricate in the setting of higher-rank graphs.
\begin{thm}
\label{thm-atomic-repres}
For a row-finite, source-free $k$-graph $\Lambda$, let $\{t_\lambda\}_{\lambda \in \Lambda}, \{\tilde{t}_{\lambda}\}_{\lambda \in \Lambda}$ generate purely atomic representations of $C^*(\Lambda)$, with associated projection valued measures $P,\tilde{P}$.    Then the two representations are unitarily equivalent if and only if the following conditions are satisfied:
\begin{enumerate}
\item[(a)] $\text{supp}(P)\;=\;\text{supp}(\tilde{P}) =: \Omega;$
\item[(b)] For every $x\in \Omega,\; \text{dim}[\text{Range}(P(\{x\}))]\;=\;\text{dim}[\text{Range}(\tilde{P}(\{x\}) )].$  
\end{enumerate}
\end{thm}
\begin{proof}
Suppose that the purely atomic representations $\{t_\lambda\}_{\lambda \in \Lambda}, \{t'_{\lambda}\}_{\lambda \in \Lambda}$ on the same Hilbert space ${\mathcal H}$ are unitarily equivalent.  Proposition \ref{propintertwine} then implies that $P, \tilde P$ have the same support $\Omega$, and moreover that the intertwining unitary takes $P(\{\omega\})$ to $\tilde P(\{\omega\})$ for every $\omega \in \Omega$.

Now, suppose that conditions $(a)$ and $(b)$ hold; we will show that the representations $\{t_\lambda\}_{\lambda \in \Lambda}$ and  $\{\tilde{t}_{\lambda}\}_{\lambda \in \Lambda}$ of $C^*(\Lambda)$ are unitarily equivalent.

Without loss of generality, we suppose that our representations are irreducible, so that $\Omega$  consists of a single orbit, $\Omega=\text{Orbit}(\omega)$.  Since
$$\text{dim}[\text{Range}(P(\{\omega\}))]\;=\;\text{dim}[\text{Range}(\tilde{P}(\{\omega\}))]$$
{by hypothesis,}
 there is a unitary isomorphism $U_{\omega}: \text{Range}(P(\{\omega\}))\;\to\;\text{Range}(\tilde{P}(\{\omega\})),$ since Hilbert spaces of the same dimension are isomorphic. 
 For every $\gamma\in \Omega=\text{Orbit}(\omega),$
we now construct a unitary $U_{\gamma}: \text{Range}(P(\{\gamma\})) \to \text{Range}(\tilde{P}(\{\gamma\}))$  as follows.
If $\gamma \in \text{Orbit}(\omega)\subset \Lambda^{\infty}$ satisfies $\gamma = a \sigma^j(\omega) $ 
 for some $a \in \Lambda$, we would like to define 
$U_{\gamma}: \text{Range}(P(\{\gamma\})) \to \text{Range}(\tilde{P}(\{\gamma\}))$
 by
\begin{equation}\label{eq:unitary}
U_{\gamma}:=\; \tilde{t}_a\tilde{t}_{\omega(0,j)}^*U_{\omega}t_{\omega(0,j)}t_a^*.
\end{equation}
We must check that $U_{\gamma}$ is well-defined and unitary, and {that \[U := \bigoplus_{\gamma \in Orbit(\omega)} U_\gamma\]} intertwines the representations.

To see that $U_{\gamma}$ is well-defined, suppose that $\gamma = a \sigma^j(\omega) = a' \sigma^{j'}(\omega)$. Fix $\xi\in {\mathcal H}$ and $\varepsilon>0.$
Since the projections $P(Z(\omega(0,n))$ tend to $P(\{\omega\})$ in the strong operator topology, it is possible to find $N_1\in \mathbb N$ (depending on $\xi\in{\mathcal H}$ and $\varepsilon>0$) such that, if we write ${\bf 1} = (1, \ldots 1) \in \N^k$, then  $ N_1 {\bf 1}  \geq j, j'$ and 
whenever $N\geq N_1,$
\begin{equation}\|P(\{\omega\})t_{\omega(0,j)}t_a^*(\xi)-P(Z(\omega(0,N {\bf 1}))t_{\omega(0,j)}t_a^*(\xi)\| <\varepsilon.\label{eq:zeroth}
\end{equation}

Write $A\;=\;a\omega(j, N {\bf 1}).$ Note that $\gamma(0,N{\bf 1})=A$ and 
$$t_{\omega(0, N {\bf 1})}t_{\omega(0, N{\bf 1})}^*t_{\omega(0,j)}t_a^* 
=\;t_{\omega(0, N{\bf 1})}t_{\omega(j, N {\bf 1})}^*t_a^*=t_{\omega(0, N {\bf 1})}t_A^*.$$
Consequently, \eqref{eq:zeroth} implies that for $N\geq N_1,$ we have 
\begin{equation}\| P(\{\omega\}) t_{\omega(0,j)} t_a^*(\xi) - t_{\omega(0, N{\bf 1})} t_A^*(\xi)\| 
 <\varepsilon\label{eq:first}
\end{equation}
and
\begin{align*}
\| U_{\omega} t_{\omega(0, N{\bf 1})}t_A^*(\xi)\;& -\;U_{\omega}P(\{\omega\})t_{\omega(0,j)}t_a^*(\xi)\|  =\| U_{\omega} t_{\omega(0, N{\bf 1})}t_{\omega(j, N{\bf 1})}^*t_a^*(\xi) \;-\;U_{\omega}P(\{\omega\})t_{\omega(0,j)}t_a^*(\xi)\|\\
& =\|U_{\omega}t_{\omega(0, N{\bf 1})}t_{\omega(0, N {\bf 1})}^*t_{\omega(0,j)}t_a^*(\xi)\;-\;U_{\omega}P(\{\omega\}))t_{\omega(0,j)}t_a^*(\xi)\| \\
&=\|U_{\omega}P(Z(\omega(0,N  {\bf 1})))t_{\omega(0,j)}t_a^*(\xi)-U_{\omega}P(\{\omega\}))t_{\omega(0,j)}t_a^*(\xi)\|<\varepsilon.\end{align*}
In the same way, we can find $N_2$ large enough so that for the same $\xi\in {\mathcal H}$ and the same $\varepsilon>0,$ for all 
$N\geq N_2,$
$$\|\tilde{P}(\{\omega\})U_{\omega}t_{\omega(0,j)}t_a^*(\xi)-\tilde{P}(Z(\omega(0,N{\bf 1})))U_{\omega}t_{\omega(0,j)}t_a^*(\xi)\|<\varepsilon.$$
Then, since $\tilde{t}_a$ and $\tilde{t}_{\omega(0,j)}$ are partial isometries,
$$\|\tilde{t}_a\tilde{t}_{\omega(0,j)}^* \tilde{P}(\{\omega\})U_{\omega}t_{\omega(0,j)}t_a^*(\xi)-\tilde{t}_a\tilde{t}_{\omega(0,j)}^*\tilde{P}(Z(\omega(0,N{\bf 1})))U_{\omega}t_{\omega(0,j)}t_a^*(\xi)\|<\varepsilon.$$
We now write, for $N\geq N_2,$
\begin{equation*}\begin{split}
\tilde{t}_a\tilde{t}_{\omega(0,j)}^*\tilde{P}(Z(\omega(0,N{\bf 1})))U_{\omega}t_{\omega(0,j)}t_a^*\; &=\;\tilde{t}_a\tilde{t}_{\omega(0,j)}^*\tilde{t}_{\omega(0,N{\bf 1})}\tilde{t}_{\omega(0,N{\bf 1})}^*U_{\omega}t_{\omega(0,j)}t_a^*\\
&\;=\;\tilde{t}_a\tilde{t}_{\omega(j,N{\bf 1})}\tilde{t}_{\omega(0,N{\bf 1})}^*U_{\omega}t_{\omega(0,j)}t_a^*\\
&\;=\;\tilde{t}_A\tilde{t}_{\omega(0,N{\bf 1})}^*U_{\omega}t_{\omega(0,j)}t_a^*.
\end{split}\end{equation*}
Therefore for $N\geq 
N_2$ we have
\begin{equation}\|\tilde{t}_a\tilde{t}_{\omega(0,j)}^* \tilde{P}(\{\omega\})U_{\omega}t_{\omega(0,j)}t_a^*(\xi)-\tilde{t}_A\tilde{t}_{\omega(0,N{\bf 1})}^* U_{\omega}t_{\omega(0,j)}t_a^*(\xi)\|<\varepsilon.\label{eq:tech-1}
\end{equation}

Since $U_\omega P(\{\omega\}) = U_\omega$ and $\tilde{P}(\{\omega\}) U_\omega = U_\omega$, Equations \eqref{eq:tech-1} and \eqref{eq:first} combine to give 
%
$$ \|\tilde{t}_a\tilde{t}_{\omega(0,j)}^* U_{\omega}t_{\omega(0,j)}t_a^*(\xi)-\tilde{t}_A\tilde{t}_{\omega(0,N{\bf 1})}^* U_{\omega} t_{\omega(0, N\cdot {\bf 1})}t_A^*(\xi)\|<2\varepsilon$$
whenever $N\geq \text{max}\{N_1,N_2\}.$

In the same way, for the same $\varepsilon > 0$ and $\xi \in \mathcal H$, we can find $M'\in \N$ (depending on $a',\; j',$ and $\xi\in {\mathcal H}$) such that {for any $N' \geq M'$,} setting 
$A'=a'\omega(j', N' {\bf 1}),$ we have  $\gamma(0,N')=A'$ and 
$$\|\tilde{t}_{a'}\tilde{t}_{\omega(0,j')}^*
 U_{\omega}t_{\omega(0,j')}t_{a'}^*(\xi)-\tilde{t}_{A'}\tilde{t}_{\omega(0,N'{\bf 1})}^* U_{\omega} t_{\omega(0, N'\cdot {\bf 1})}t_{A'}^*(\xi)\|<2\varepsilon.$$
{Choosing  $N = N' \geq \max \{ M', N_1, N_2\}$ implies that $A = \gamma(0, N {\bf 1}) = A'$.  Thus, }
\begin{equation*}
\|\tilde{t}_a\tilde{t}_{\omega(0,j)}^* 
U_{\omega}t_{\omega(0,j)}t_a^*(\xi)-\tilde{t}_{a'}\tilde{t}_{\omega(0,j')}^* 
U_{\omega}t_{\omega(0,j')}t_{a'}^*(\xi)\|<4\varepsilon.\end{equation*}
Since $\varepsilon$ {and $\xi$} were arbitrary, it follows that  if $\gamma = a \sigma^j(\omega) = a' \sigma^{j'}(\omega)$,
\[\tilde{t}_a \tilde{t}_{\omega(0, j)}^* U_{\omega}t_{\omega(0,j)}t_a^*=\tilde{t}_{a'}\tilde{t}_{\omega(0,j')}^* 
U_{\omega}t_{\omega(0,j')}t_{a'}^*.\]
Thus, the operator $U_{\gamma} : \text{Range}(P(\{\gamma\})) \to \text{Range}(\tilde{P}(\{\gamma\}))$ of Equation \eqref{eq:unitary}
is well-defined.

Now we want to show that $U_{\gamma}$ is unitary. Since $U_\omega$ is a unitary and hence $U_\omega^* U_\omega = P(\{\omega\})$, using the facts that $ \tilde P(\{\omega\}) U_\omega = U_\omega$ and $\tilde P(\{\omega\}) \tilde P(Z(\omega(0,j))) = \tilde P(\{\omega\})$, one easily computes that $U_\gamma ^* U_\gamma = t_a t_{\omega(0,j)}^* P(\{\omega\}) t_{\omega(0,j)} t_a^*$.

We now note that for $\gamma= a \sigma^j(\omega),$ 
$t_a^*$ takes $\text{Range}(P(\{\gamma\}))$ to $ \text{Range}(P(\{\sigma^j(\omega)\}))$
and 
$t_{\omega(0,j)}$ takes $\text{Range}(P(\{\sigma^j(\omega)\}))$ to $\text{Range}(P(\{\omega\})).$ {Recalling that $U_\gamma = U_\gamma P(\{\gamma\})$, we deduce that 
\begin{align*} U_\gamma^* U_\gamma &= t_a t_{\omega(0,j)}^* t_{\omega(0,j)} t_a^* P(\{\gamma\}) = t_a t_a^* P(\{\gamma\}) \\
&= P(Z(a)) P(\{\gamma\}) = P(\{\gamma\}).
\end{align*}
}
Similarly, one can show that
$$U_{\gamma}U_{\gamma}^*\;=\;\tilde{P}(\{\gamma\}),$$ which implies that $U_\gamma$ is unitary from its domain to its range.

To show that $U= \bigoplus_{\gamma \in Orbit(\omega)} U_\gamma$ intertwines the representations, we must establish that for $\lambda\in\Lambda$ with $s(\lambda)=r(\omega)$,
$$\tilde{t}_{\lambda}U_{\omega}\;=\;U_{\lambda\omega}t_{\lambda}.$$
By our construction of $U_{\lambda \omega}$, if $s(\lambda)\;=\; r(\omega),$
\[
U_{\lambda\omega}t_{\lambda}\;=\;\tilde{t}_{\lambda}U_{\omega}t_{\lambda}^*t_{\lambda}.
\]
Using the fact that $t_\lambda$ is an isometry {and that $U_\omega = U_\omega P(\{\omega\}) = U_\omega P(\{\omega\}) P(Z(\omega(0,n)))$ for any $n \in \N^k$,} we obtain
\[\begin{split}
\tilde{t}_{\lambda} U_{\omega}t_{\lambda}^*t_{\lambda}&=
\tilde{t}_\lambda U_\omega P(Z(s(\lambda)))=\tilde{t}_\lambda U_\omega P(Z(r(\omega)))=\tilde{t}_\lambda U_\omega.
\end{split}\]


Therefore we see that 
$$U= \bigoplus_{\gamma\in\;\text{Orbit}(\omega)}U_{\gamma}:\;\bigoplus_{\gamma\in\;\text{Orbit}(\omega)}\text{Range}(P(\{\gamma\}))\;\to\; \;\bigoplus_{\gamma\in\;\text{Orbit}(\omega)}\text{Range}(\tilde{P}(\{\gamma\}))$$
is a unitary operator that intertwines the representations $\{t_{\lambda}: \lambda\in \Lambda\}$ and $\{\tilde{t}_{\lambda}: \lambda\in \Lambda\}.$
\end{proof}

Recall that two representations $\pi, \pi'$ of a $C^*$-algebra $A$ are  \emph{disjoint} if  no nonzero
subrepresentation of $\pi$ is unitarily equivalent to a subrepresentation of $\pi'$.

{The following proposition shows that different orbits support disjoint representations.  It also characterizes the intertwiners of those purely atomic representations which are supported on the same orbit, and shows when a purely atomic representation is irreducible.}

\begin{prop}\label{prop:irr-rank1}
For a row-finite, source-free $k$-graph $\Lambda,$  let $\{t_\lambda\}_{\lambda \in \Lambda}, \{\tilde{t}_{\lambda}\}_{\lambda \in \Lambda}$ generate purely atomic representations of $C^*(\Lambda)$, with associated projection valued measures $P, \tilde{P}$.  Suppose that $P, \tilde P$ are supported on Orbit$(\gamma)$ and Orbit$(\omega)$ respectively, for some $\gamma, \omega \in \Lambda^\infty$.  
\begin{enumerate}
\item[(a)] If Orbit$(\gamma) \not= \text{Orbit}(\omega)$ then the representations are disjoint.
\item[(b)] If Orbit$(\gamma) = \text{Orbit}(\omega)$, then the operators   $X: {\mathcal{H}} \to \tilde{\mathcal{H}}$ which intertwine the representations $\{t_\lambda\}_{\lambda\in\Lambda}$ (on $\mathcal{H}$) and  $\{\tilde{t}_\lambda\}_{\lambda\in\Lambda}$ (on $\tilde{\mathcal{H}}$) are in bijective correspondence with operators  $Y_\omega: \Ran ( P(\{\omega\})) \to \Ran( \tilde{P}(\{\omega\}))$, via the formula 
\[
Y_\omega=\tilde{P}(\{\omega\})X P(\{\omega\}).
\]
$Y_\omega: \Ran (P(\{\omega\})) \to \Ran(\tilde{P}(\{\omega\}))$.
\item[(c)] The representation  $\{t_\lambda\}_{\lambda\in\Lambda}$ is irreducible if and only if  $\dim[ \Ran (P(\{\omega\}))]= 1.$ 
\end{enumerate}
\end{prop}
\begin{proof}
{To see (a), note that under our hypotheses,  neither $\{t_\lambda\}_\lambda$ nor $\{\tilde t_\lambda\}_\lambda$ has nontrivial subrepresentations; therefore Theorem \ref{thm-atomic-repres} implies (a).}
For (b), 
given an operator $X: \H \to\H'$ which intertwines the representations, 
the operator $Y_\omega:=  \tilde P(\{ \omega\}) X P(\{\omega\})$
is evidently a well-defined operator from $\Ran (P(\{\omega\}))$ to $\Ran (\tilde P(\{\omega\}))$. 
On the other hand, given an operator $Y_\omega: \Ran (P(\{\omega\})) \to \Ran (\tilde P(\{\omega\}))$, we can  define $X: \H \to \tilde \H$ by setting
\[ X|_{\Ran(P(\{\gamma\}))} = \tilde t_a \tilde  t_{\omega(0,j)}^* Y_\omega t_{\omega(0,j)} t_a^*\]
whenever $\gamma = a\sigma^j(\omega) \in \text{Orbit}(\omega)$. Arguments analogous to those  employed in the proof of Theorem \ref{thm-atomic-repres} will show that $X$ is well defined; we omit the details.

{For (c),} let  ${\mathcal H}=\text{Range}\,P(\{\omega\}),$ and $\tilde{\mathcal H}= \text{Range}\,\tilde{P}(\{\omega\}).$  Recall from Theorem \ref{thm-atomic-repres} above that if ${\mathcal H}$ and $\tilde{\mathcal H}$ have the same dimension,  any unitary $U_{\omega}\in {\mathcal U}({\mathcal H},\tilde{\mathcal H})$ can be used to construct an intertwiner 
$$\bigoplus_{\gamma\in\;\text{Orbit}(\omega)}U_{\gamma}:\;\bigoplus_{\gamma\in\;\text{Orbit}(\omega)}\text{Range}\,P(\{\gamma\})\;\to\; \bigoplus_{\gamma\in\;\text{Orbit}(\omega)}\text{Range}\,\tilde{P}(\{\gamma\}).$$
In  the same way, if $T_{\omega}$ is a finite linear combination of unitary elements in $B({\mathcal H})=B(\text{Range}\,P(\{\omega\})),$ defining for $\gamma=a\sigma^j(\omega)$ the bounded operator 
$$T_{\gamma}=\;t_a t_{\omega(0,j)}^*T_{\omega}t_{\omega(0,j)}t_a^*,$$
Theorem \ref{thm-atomic-repres} shows us that $T_{\gamma}$ is well-defined and that 
$$\bigoplus_{\gamma\in\;\text{Orbit}(\omega)}T_{\gamma}:\;\bigoplus_{\gamma\in\;\text{Orbit}(\omega)}\text{Range}\,P(\{\gamma\})\;\to\; \bigoplus_{\gamma\in\;\text{Orbit}(\omega)}\text{Range}\,P(\{\gamma\})$$
intertwines the representation  $\{t_\lambda\}_{\lambda\in\Lambda}$ with itself.

   Consequently, the set of intertwiners between  $\{t_\lambda\}_{\lambda\in\Lambda}$ and itself contains the span of the unitary elements in 
$B({\mathcal H})=B(\text{Range}\,P(\{\omega\})).$  
However, by Russo-Dye's Theorem, the closure of the span of the unitary elements in $B({\mathcal H})$ is exactly $B({\mathcal H}).$   By Schur's Lemma, our representation is irreducible if and only if the self-intertwiners of   $\{t_\lambda\}_{\lambda\in\Lambda}$ consist solely of scalar multiples of the identity.  But by our preceding construction, we see that this happens if and only if the dimension of ${\mathcal H}=\text{Range}\,P(\{\omega\})$ is equal to $1,$ as desired.
\end{proof}

\subsection{Purely atomic representations versus monic representations}
\label{sec:relation_monic_atomic}

We now discuss the relation between purely atomic representations and monic representations.

Recall first (cf.~\cite[Definition 4.1]{FGJKP-monic}) that a representation $\{ t_\lambda\,:\, \lambda\in\Lambda\}$ of a finite source-free $k$-graph  on a Hilbert space $\H$ is called \emph{monic} if $t_\lambda \not= 0$ for all $\lambda \in \Lambda$, and there exists a vector $\xi\in \mathcal{H}$ such that
 \[
 \overline{\text{span}}_{\lambda \in \Lambda} \{ t_\lambda t_\lambda^* \xi \} = \mathcal{H}.
 \]

According to Theorem~4.2 of \cite{FGJKP-monic}, for a finite source-free $k$-graph $\Lambda$, every monic representations of $C^*(\Lambda)$ is unitarily equivalent to a representation of $C^*(\Lambda)$ on $L^2(\Lambda^\infty, \mu)$ for some Borel measure $\mu$.   
 The next theorem proves that a purely atomic representation of $C^*(\Lambda)$ is monic if and only if for every atom $x\in \Lambda^\infty$, $P(\{x\})$ is one-dimensional. 

\begin{rmk}
 Since $L^2(\Lambda^\infty, \mu)$ is separable for any measure $\mu$ associated to a monic representation, in the setting of Theorem \ref{thm:atomic-1D}
we conclude that the set of atoms for $\mu$ must be countable.
\end{rmk}

\begin{thm}
\label{thm:atomic-1D}
Let $\Lambda$ be a finite $k$-graph with no sources.
Let $\{t_\lambda:\lambda\in \Lambda\}$ be a purely atomic representation of $C^*(\Lambda)$ on a separable Hilbert space $\mathcal{H}$. Suppose that $t_\lambda t^*_\lambda\ne 0$ for all $\lambda\in\Lambda$. Then the representation is monic if and only if for every atom $x\in \Lambda^\infty$, $P(\{x\})$ is one-dimensional. Moreover, in this case the associated measure $\mu$ arising from the monic representation is atomic.

\end{thm}

\begin{proof}
Suppose that the given purely atomic representation $\{t_\lambda:\lambda\in \Lambda\}$ on $\mathcal{H}$ is monic, with cyclic vector $\xi$ for $\{ t_\lambda t_\lambda^*\}_{\lambda \in \Lambda}$. Then by Theorem~4.2 of \cite{FGJKP-monic} we can assume that $\mathcal{H}$ is of the form $L^2(\Lambda^{\infty},\mu)$, where the measure $\mu$ is given by the projection valued measure $P$ determined by the representation, i.e. $\mu(Z(\lambda))=\langle P(Z(\lambda)) \xi, \xi \rangle = \| t_\lambda^* \xi  \|^2$ for $\lambda\in\Lambda$.
Since $\{t_\lambda\}_\lambda$ is purely atomic, $\mu(\{\omega\}) = \|P(\{\omega\})\xi\|^2$ is nonzero iff $\omega \in \Omega$.  In other words, the atoms of $\mu$ are precisely the atoms of $P$.

To show that $P(\{\omega\})$ is always a rank-one projection for an atom $\omega$, we argue by contradiction.  Suppose that there exists $\omega \in \Lambda^\infty$ and a strict subprojection $Q_\omega \leq P(\{\omega\})$ with $Q_\omega \not= P(\{\omega\})$.  For any $\gamma = a \sigma^j(\omega) \in \text{Orbit}(\omega)$, write 
\[ Q_\gamma = t_a t_{\omega(0,j)}^* Q_\omega t_{\omega(0,j)} t_a^*,\] and set $Q = \bigoplus_{\gamma \in \text{Orbit}(\omega)} Q_\gamma$.

The fact that the projections $P(\{\gamma\})$ are mutually orthogonal  implies that $Q$ is indeed a sum of orthogonal projections.  Moreover, Proposition \ref{prop:atomic-2} implies that
each summand $Q_\gamma$ is a strict subprojection of $P(\{\gamma\})$.

We will show that $t_\eta Q  = Qt_\eta$ for all $\eta \in \Lambda$.
Since $\{t_\lambda\}_{\lambda \in \Lambda}$ is monic by assumption,   Theorem~3.13 and Theorem~4.2 of \cite{FGJKP-monic} will then imply that $Q$ must be a multiplication operator, which contradicts the fact that each $Q_\gamma$ is a strict subprojection of $P(\{\gamma\})$.

Fix $ \eta \in \Lambda$ and $\gamma = a \sigma^j(\omega)$.  As in the proof of Proposition \ref{prop-atomic-basic-equns}, 
\[ Q_\gamma  t_\eta = t_a t_{\omega(0,j)}^*  Q_\omega t_{\omega(0,j)} t_a^* t_\eta =\sum_{(\rho, \zeta) \in \Lambda^{\min}(a, \eta)} t_a t_{\omega(0,j)}^*  Q_\omega t_{\omega(0,j)} t_\rho t_\zeta^*.\]
By Proposition \ref{prop:atomic-2}, $t_{\omega(0,j)}^* Q_\omega t_{\omega(0,j)} = Q_{\sigma^j(\omega)} = Q_{\sigma^j(\omega)} P(Z(\omega(j, j +d(\rho)))),$ and $t_{\omega(j, j+d(\rho))}^* t_\rho =0$ unless $\rho = \omega(j, j+d(\rho))$.  Thus, the sum collapses to (at most) a single term:  Writing $m = d(\rho) = d(a) \vee d(\eta) - d(a)$, 
\[ Q_\gamma t_\eta = \begin{cases}
t_a Q_{\sigma^j(\omega)} t_{\omega(j, j+m)} t_\zeta^*,& \eta \zeta = a \omega(j, j+m)\\
0 ,& \eta \zeta \not= a \omega(j, j+m)
\end{cases}\]
Now, using the fact that $Q_{\sigma^j(\omega)} = t_{\omega(j, j+m)} t_{\omega(j,j+m)}^* Q_{\sigma^j(\omega)}$, we obtain that if $Q_\gamma t_\eta \not= 0$, 
\[ Q_\gamma t_\eta = t_{\eta \zeta} t_{\omega(j,j+m)}^* Q_{\sigma^j(\omega)} t_{\omega(j,j+m)} t_\zeta^* = t_\eta Q_{\zeta \sigma^{m+j}(\omega)}.\]
For each fixed $\eta$, the map 
\[a \sigma^j(\omega) \mapsto\zeta \sigma^{m+j}(\omega), \quad \text{ where }a \omega(j,m+j) =\eta \zeta,\]
is a bijection from $\{\gamma \in \text{Orbit}(\omega): Q_\gamma t_\eta \not= 0\}$ to $\{ \tilde \gamma \in \text{Orbit}(\omega): t_\eta Q_{\tilde \gamma} \not= 0\}$. (Surjectivity follows by observing that, given $\eta \in \Lambda$ and $\tilde\gamma = \zeta \sigma^{q}(\omega)$ with $s(\eta) = r(\zeta)$, we can take $a = \eta \zeta, j = q$ to construct the preimage $\gamma$ of $\tilde \gamma$.)

  It now follows that, as claimed,
\[ Q t_\eta = t_\eta Q.\]

Conversely, suppose that $\{t_\lambda:\lambda\in \Lambda\}$ is a purely atomic representation of $C^*(\Lambda)$ on a separable Hilbert space $\mathcal{H}$ such that for every atom $x\in \Lambda^{\infty},\;P(\{x\})\H$ is one-dimensional. Let $\Omega\subset \Lambda^{\infty}$ be the support of the associated projection valued measure $P$ on $\Lambda^{\infty}.$ Since $\mathcal{H}$ is separable and since $P(\{x\})\H$ is orthogonal to $P(\{y\})\H$ for $x \not= y\in \Omega,$ we must have that  $\Omega$ is countable; let us enumerate $\Omega\;=\;\{\omega_n\}_{n=1}^{\infty}.$ Then
$$\sum_{n=1}^{\infty}P(\{\omega_n\}) = Id_{\mathcal{H}},$$
where the convergence is in the strong operator topology.
 For each $n\in\mathbb N,$ choose a unit vector $e_n\in P(\{\omega_n\})\H.$ 
Define $\xi\in\H$ by
 $$\xi=\sum_{n=1}^{\infty}\frac{e_n}{2^n}.$$
 We note that $P(\{\omega_n\})(\xi)\;=\;\frac{e_n}{2^n}.$
 It follows that for each $n\in\mathbb N,$
 $$e_n\in\;\overline{\text{span}}\{t_{\lambda}t_{\lambda}^*(\xi)=P(Z(\lambda)(\xi):\;\lambda\in \Lambda\}.$$
This is due to the fact that for each $n\in\mathbb N,$ 
 $$\lim_{j\to \infty}t_{\omega_n(0,j)}t_{\omega_n(0,j)}^*(\xi)\;=\;\lim_{j\to \infty}P(Z(\omega(0,j)))(\xi)$$
 $$=\;P(\{\omega_n\})(\xi)\;=\;\frac{e_n}{2^n}.$$
 
Therefore $\xi$ is a cyclic vector for $\{t_{\lambda}t_{\lambda}^*: \;\lambda\in \Lambda\},$ so that this representation is monic.

\end{proof}

\section{Permutative representations of $C^*(\Lambda)$}
\label{sec:permutative_repn}


Here we  study permutative   representations of $C^*(\Lambda)$.  These are similar, but not precisely equivalent, to the atomic representations of single-vertex $k$-graphs studied by Davidson, Power and Yang in \cite{dav-pow-yan-atomic}.  In particular, the atomic representations of \cite{dav-pow-yan-atomic} permit for a rescaling, in addition to a permutation, of the basis vectors of the Hilbert space.   We refer the reader to that paper and the references therein for more details.

\subsection{Definition and first properties}

 \begin{defn} 
 \label{defpermutative}
(c.f. Definition 4.9 of \cite{dutkay-jorgensen-atomic}.) 
 Let $\Lambda$ be a row-finite $k$-graph with no sources. 
A representation $\{t_\lambda\}_{\lambda \in \Lambda}$ of $C^*(\Lambda)$ on a Hilbert space ${\mathcal H}$ is called {\it permutative} if ${\mathcal H}$ has an orthonormal basis $\{e_i: i\in I\}$ for some index set $I$ such that for each $\lambda\in \Lambda$ there are subsets $J_{\lambda}$ and $K_{\lambda}$ of $I$ and a bijection 
 $\tilde{\sigma}_{\lambda}:J_{\lambda}\to K_{\lambda}$ satisfying
\begin{enumerate}
 \item[(a)] For each $n\in\mathbb N^k,\;\cup_{\lambda\in \Lambda^n}J_{\lambda}\;=\;\cup_{\lambda\in \Lambda^n}K_{\lambda}=I;$
\item[(b)] For each $\lambda\in \Lambda$ and $\nu\in s(\lambda)
\Lambda,$ we have $K_{\nu}\subset J_{\lambda}$ and $\tilde{\sigma}_{\lambda}\circ \tilde{\sigma}_{\nu}=\tilde{\sigma}_{\lambda\nu}$. (This implies $J_{\lambda\nu}=J_{\nu}$ whenever $s(\lambda)=r(\nu)$).
 \item[(c)] $t_{\lambda}(e_i)\;=\;e_{\tilde{\sigma}_{\lambda}(i)}$ for $i\in J_{\lambda},$ and $t_{\lambda}(e_i)=0,$ for $i\notin J_{\lambda}.$
 \item[(d)] $t_{\lambda}^{\ast}(e_{\tilde{\sigma}_{\lambda}(i)})=e_{i}$ for $i\in J_{\lambda},$ and $t_\lambda^*(e_j) = 0$ for $j \in K_{\lambda'}$, if $\lambda \not= \lambda'$ but $d(\lambda) = d(\lambda')$.
 \end{enumerate}

 \end{defn}

An attentive reader may notice a similarity between Definition \ref{defpermutative} and the definition of a $\Lambda$-semibranching function system (\cite{FGKP} Definition 3.2; see also Theorem 3.1 of  \cite{FGJKP-SBFS} for an equivalent formulation).  Intuitively, a $\Lambda$-semibranching function system is a ``representation'' of $C^*(\Lambda)$ on a measure space $(X, \mu)$; it consists of a family of partially defined measurable maps $\{\tau_\lambda: D_{s(\lambda)} \to X\}_{\lambda \in \Lambda}$ whose range sets $R_\lambda := \tau_\lambda(D_{s(\lambda)})$ satisfy measure-theoretic analogues of the Cuntz--Krieger relations.  
We formalize the connection between permutative representations and $\Lambda$-semibranching function systems in Section \ref{sec:perm-and-sbfs} below.

We first prove:
\begin{lemma}  
 Let $\Lambda$ be a row-finite $k$-graph with no sources.
Let $\{ t_\lambda\}_{\lambda \in \Lambda}$ be a permutative representation of $C^*(\Lambda)$ on a Hilbert space ${\mathcal H},$  and let $\{J_{\lambda}\}_{\lambda\in\Lambda}$ and $\{K_{\lambda}\}_{\lambda\in\Lambda}$ be as in Definition~\ref{defpermutative}.  For any  $n\in\mathbb N^k,$ if $\lambda,\;\lambda'\in \Lambda^n$ and $\lambda\not=\lambda',$ then we have 
$K_{\lambda}\cap K_{\lambda'}=\emptyset.$
\label{lem:permutative}

\end{lemma}
\begin{proof} We recall if $\lambda,\lambda'\in\Lambda^n$ then 
$$t_{\lambda'}^{\ast}t_{\lambda}=\delta_{\lambda',\lambda}t_{s(\lambda)}.$$
So, for $\lambda,\lambda'\in\Lambda^n$ with $\lambda\not=\lambda',$ if there exists $j\in K_{\lambda}\cap K_{\lambda'},$ we could find $i\in J_{\lambda}$ with $\tilde{\sigma}_{\lambda}(i)=j,$ and $k\in J_{\lambda'}$ with $\tilde{\sigma}_{\lambda'}(k)=j.$
But then by definition of permutative representation, we would have 
$$t_{\lambda'}^{\ast}t_{\lambda}(e_i)\;=\;t_{\lambda'}^{\ast}(e_{\tilde{\sigma}_{\lambda}(i)})=t_{\lambda'}^{\ast}(e_j)\;=\;t_{\lambda'}^{\ast}(e_{\tilde{\sigma}_{\lambda'}}(k))=e_k\not=0.$$
But this contradicts the fact that $t_{\lambda'}^{\ast}t_{\lambda}=0$ for $\lambda\ne \lambda'$, so we must have $K_{\lambda}\cap K_{\lambda'}=\emptyset.$
\end{proof}
 We define the {\it  encoding map} $E$ from the index set $I$ into $\Lambda^{\infty}$
by 
\begin{equation}\label{eq:encoding}
E(i)((0,n))=\lambda,\;\text{where}\;\lambda\;\text{is the unique element of}\;\Lambda^n\;\text{such that}\; i \;\in\:K_{\lambda}.
\end{equation}
To see that $E(i) \in \Lambda^\infty$ is well defined, it suffices to check that if $m \geq n$ and $\lambda = E(i)((0,n))$, so that $i \in K_\lambda,$  then $E(i)((0,m)) = \lambda \nu$ for some $\nu \in \Lambda$.  Thus, 
suppose that $m\geq n\in \mathbb N^k.$
Write
$E(i)((0,m))=\mu.$  Then there exists a unique $j'\in J_{\mu}$ with $\tilde{\sigma}_{\mu}(j')=i\in K_{\mu}.$

Write $\mu = \lambda' \nu'$ where $d(\lambda') = n$.  We will show that $\lambda' = \lambda$.
Since $J_{\mu}=J_{\lambda'\nu'}\;=J_{\nu'}$ and 
$$\tilde{\sigma}_{\lambda'\nu'}\;=\tilde{\sigma}_{\lambda'}\circ \tilde{\sigma}_{\nu'},$$
we obtain $i = \tilde{\sigma}_{\lambda'} ( \tilde{\sigma}_{\nu'}(j')) \in K_{\lambda'}$.  Lemma \ref{lem:permutative} now implies that $\lambda = \lambda'$.

\begin{prop}\label{prop:encoding-map}
Let $\Lambda$ be a row-finite $k$-graph with no sources and let $I$ be an index set associated to a permutative representation  $\{ t_\lambda\}_{\lambda \in \Lambda}$  of $C^*(\Lambda)$.  Suppose that $E:I\to \Lambda^{\infty}$ is the encoding map given in \eqref{eq:encoding}. Then we have the following.
\begin{itemize}
\item[(a)]For each $i\in I$ and $n\in\mathbb N^k$, there is a unique $\lambda\in \Lambda^n$ and $i_n\in\;J_{\lambda}\subset I$ such that  $\tilde{\sigma}_{\lambda}(i_n)=i.$  Writing $\tilde{\sigma}^n(i)=i_n,$  we have $\tilde{\sigma}_{\lambda}\circ \tilde{\sigma}^n(i)=i$ for all $i\in K_{\lambda},$ and $\tilde{\sigma}^n\circ \tilde{\sigma}_{\lambda}(i)=i$ for all $i\in J_{\lambda}.$  

\item[(b)] The map $E:I\to\Lambda^{\infty}$ defined above satisfies 
\[
\sigma_{\lambda}(E(i))\;=\;E(\tilde{\sigma}_{\lambda}(i))\quad \text{for $i\in J_{\lambda}$},\quad\text{and}\quad \sigma^n(E(i))\;=\;E(\tilde{\sigma}^n(i))\quad \text{for $i\in I$}.
\]
\end{itemize}
\end{prop}
\begin{proof} For (a), notice that for $n\in \mathbb N^k,$ we have $I=\bigsqcup_{\lambda\in \Lambda^n}K_{\lambda},$ 
so that fixing $i\in I,$  there is a unique $\lambda\in \Lambda^n$ such that  $i\in K_{\lambda}.$ Since $\tilde{\sigma}_{\lambda}$ is a bijection from $J_{\lambda}$ to $K_{\lambda},$ there is a unique $i_n\in J_{\lambda}$ such that $\tilde{\sigma}_{\lambda}(i_n)=i$. 
Also, by definition of $\tilde{\sigma}^n,$ we have that  $\tilde{\sigma}_{\lambda}\circ \tilde{\sigma}^n(i)\;=\;\tilde{\sigma}_{\lambda}(i_n)=i$ for $i\in K_{\lambda},$ and  similarly 
$\tilde{\sigma}^n\circ \tilde{\sigma}_{\lambda}(i)=i,$ for all $i\in J_{\lambda}.$

For (b), recall $ \omega\in \Lambda^{\infty}$ is in the domain of $\sigma_{\lambda}$ if and only if $r(\omega)=s(\lambda).$  So recalling that $E(i)((0,0))$ is the unique $v\in \Lambda^0$ such that $i\in J_v,$ we will have $\sigma_{\lambda}(E(i))$ is defined if and only if $s(\lambda)=E(i)((0,0)),$ i.e.~$i\in J_{s(\lambda)}=J_{r(E(i))}.$ 

Recall $\tilde{\sigma}_{\lambda}(i)$ is defined only when $i\in J_{\lambda},$  and that Condition (b) of  Definition \ref{defpermutative} implies that  $J_{\lambda}=J_{s(\lambda)}\subset I.$  Also,  if $i\in J_{\lambda},$ we have $\tilde{\sigma}_{\lambda}(i)\in K_{\lambda},$ and then
$$E( \tilde{\sigma}_{\lambda}(i))((0,d(\lambda)))=\lambda.$$
  On the other hand, for any $\omega \in s(\lambda)\Lambda^\infty$, we have $\sigma_\lambda(\omega) (0, d(\lambda)) = \lambda$. In particular,
$$\sigma_{\lambda}( E(i))((0,d(\lambda)))=\;\lambda.$$
 We thus can see if $n'\in \mathbb N^k$ and $n'\leq d(\lambda),$
$$\sigma_{\lambda}(E(i))((0,n'))=E(\tilde{\sigma}_{\lambda}(i))((0,n'))=\lambda(0, n').$$

Now suppose $m\in\mathbb N^k,\;m\not=n;= d(\lambda).$  Let $\ell=m\vee n$, the coordinatewise maximum of $m$ and $n$.	
Then $n\leq \ell$ and $m\leq \ell.$  Find $n',\;m'\in \mathbb N^k$ such that $n+n'=m+m'=\ell.$ 
Write $\eta :=\sigma_{\lambda}(E(i))((0,\ell));$
by the factorization property, we can find $\lambda',\gamma, \gamma'\in \Lambda$ with $d(\lambda')=n',\;d(\gamma)=m,$ and $d(\gamma')=\; m'$ such that 
$\eta\;=\;\lambda\lambda'\;=\;\gamma\gamma'.$
Then
$$\eta = \sigma_{\lambda}(E(i))((0,\ell))=\sigma_{\lambda}(E(i))((0,n+n'))= \lambda \, E(i)((0, n')).$$
Similarly, $E(i)((0, n')) =\lambda'$ and $\sigma_{\lambda}(E(i))((0,m))\;=\;\gamma$.
Now, we observe that
$$E(\tilde{\sigma}_{\lambda}(i))((0,\ell))=E( \tilde{\sigma}_{\lambda}(i))((0,n+n'))=\lambda \lambda'=\eta =\gamma \gamma’ = E(\tilde \sigma_\lambda(i)) ((0, m+m’)).$$
To see this, we observe that $\eta$ is the unique element of $\Lambda^\ell$ such that $\tilde \sigma_\lambda(i) \in K_\eta = K_{\lambda \lambda’};$ this follows from the fact that $\tilde \sigma_{\lambda’} \circ \tilde \sigma_\lambda = \tilde \sigma_\eta.$ In other words,
$$( \tilde{\sigma}_{\lambda}(i))((0,\ell))\;=\; E( \tilde{\sigma}_{\eta}(i'))((0,\ell))\;=\;\eta.$$
Since $\eta=\gamma \gamma'$ with $d(\gamma)=m,$ we also have  
$$E(\tilde{\sigma}_{\lambda}(i))((0,m))\;=\; E( \tilde{\sigma}_{\eta}(i'))((0,m))\;=\gamma.$$
It follows that for all $m\in \mathbb N^k,$
$$\sigma_{\lambda}(E(i))((0,m))\;=\;\gamma\;=\;E( \tilde{\sigma}_{\lambda}(i))((0,m)),$$
proving the first equality of (b).

From this, we will deduce the second equality.
Let $i\in I,\; n\in \mathbb N^k,$ and suppose that $E(i)((0,n))=\lambda\in \Lambda^n.$
Then setting $i_n=\tilde{\sigma}^n(i),$ we have $i_n\in J_{\lambda},\;i=\; \tilde{\sigma}_{\lambda}\circ \tilde{\sigma}^n(i)\;=\;\tilde{\sigma}_{\lambda}(i_n)\;\in K_{\lambda}.$ Moreover, the first equality of (b) gives
$$\sigma_{\lambda}( E(i_n))\;=\;E(\tilde{\sigma}_{\lambda}(i_n)).$$
We now apply $\sigma^n$ to both sides of this equation to obtain
$$\sigma^n\circ\;\sigma_{\lambda}\circ E(i_n)\;=\;\sigma^n\circ E\circ \tilde{\sigma}_{\lambda}(i_n),$$
and consequently 
$E\circ \tilde{\sigma}^n(i)\;=\;\sigma^n\circ E(i).$
\end{proof}
When $E$ is injective, we obtain the following corollary.
\begin{cor}
Let $\Lambda$ be a row-finite $k$-graph with no sources.
 Let $\{ S_\lambda\}$ be a permutative representation of $C^*(\Lambda)$ on a Hilbert space ${\mathcal H},$ and let $\{e_i:i\in I\}$ be the ``permuted" basis for ${\mathcal H}.$  Let $E:I\to \Lambda^{\infty}$ be the encoding map of Equation \eqref{eq:encoding}.  Then if $E$ is one-to-one, the set $I$ can be identified with a subset of infinite paths $\Omega := E(I)$ in $\Lambda^{\infty}$ and the maps $\{\tilde{\sigma}_{\lambda}:\lambda\in\Lambda\}$ and $\{\tilde{\sigma}^n\;:n\in\mathbb N^k\}$ can be identified with the corresponding shifts and coding maps on the subset $E(I)=:\Omega$ of $\Lambda^{\infty}.$
\end{cor}
\begin{proof}
Since $E$ is a bijection from $I$ onto $\Omega\subset \Lambda^{\infty},$ the map $E^{-1}:\Omega\to I$ is well-defined, and we obtain 
from Proposition~\ref{prop:encoding-map}
\[
E^{-1}\circ \sigma_{\lambda}\circ E(i)\;=\;\tilde{\sigma}_{\lambda}(i)\quad\text{for $i\in J_{\lambda}$},\quad\text{and}\quad E^{-1}\circ\sigma^n\circ E(i)\;=\; \tilde{\sigma}^n(i)\quad \text{for $i\in I$.} \qedhere
\]
\end{proof}

\subsection{Permutative representations and $\Lambda$-semibranching function systems}
\label{sec:perm-and-sbfs}
The objective of this section is to prove, in Proposition \ref{prop:permutative-are-sbfs},  that any permutative representation of a finite source-free $k$-graph arises from a $\Lambda$-semibranching function system on a discrete measure space.  We begin by recalling the relevant definitions.
\begin{defn}
\label{def-1-brach-system}\cite[Definition~2.1]{MP}\label{defn:sbfs}
Let $(X,\mu)$ be a measure space. Suppose that, for each $1\le i\le N$, we have a measurable map $\sigma_i:D_i\to X$, for some measurable subsets $D_i\subset X$. The family $\{\sigma_i\}_{i=1}^N$ is a \emph{semibranching function system} if the following holds:
\begin{itemize}
\item[(a)] Setting $R_i = \sigma_i(D_i),$ we have 
\[
\mu(X\setminus \cup_i R_i)=0,\quad\quad\mu(R_i\cap R_j)=0\;\;\text{for $i\ne j$}.
\]
\item[(b)] For each $i$, the Radon--Nikodym derivative
\[
\Phi_{\sigma_i}=\frac{d(\mu\circ\sigma_i)}{d\mu}
\]
satisfies $\Phi_{\sigma_i}>0$, $\mu$-almost everywhere on $D_i$.
\end{itemize}
A measurable map $\sigma:X\to X$ is called a \emph{coding map} for the family $\{\sigma_i\}_{i=1}^N$ if $\sigma\circ\sigma_i(x)=x$ for all $x\in D_i$.
\end{defn}

\begin{defn}\cite[Definition~3.2]{FGKP}
\label{def-lambda-SBFS-1}
Let $\Lambda$ be a finite $k$-graph and let $(X, \mu)$ be a measure space.  A \emph{$\Lambda$-semibranching function system} on $(X, \mu)$ is a collection $\{D_\lambda\}_{\lambda \in \Lambda}$ of measurable subsets of $X$, together with a family of prefixing maps $\{\tau_\lambda: D_{\lambda} \to X\}_{\lambda \in \Lambda}$, and a family of coding maps $\{\tau^m: X \to X\}_{m \in \N^k}$, such that
\begin{itemize}
\item[(a)] For each $m \in \N^k$, the family $\{\tau_\lambda: d(\lambda) = m\}$ is a semibranching function system, with coding map $\tau^m$.
\item[(b)] If $ v \in \Lambda^0$, then  $\tau_v = id$,  and $\mu(D_v) > 0$.
\item[(c)] Let $R_\lambda = \tau_\lambda( D_{\lambda})$. For each $\lambda \in \Lambda, \nu \in s(\lambda)\Lambda$, we have $R_\nu \subseteq D_{\lambda}$ (up to a set of measure 0), and
\[\tau_{\lambda} \tau_\nu = \tau_{\lambda \nu}\text{ a.e.}\]
\item[(d)] The coding maps satisfy $\tau^m \circ \tau^n = \tau^{m+n}$ for any $m, n \in \N^k$.  
\end{itemize}
\end{defn}
Observe that Condition  (c) implies that $D_\lambda = D_{s(\lambda)}$ for any $\lambda \in \Lambda$, and Condition (d) forces the coding maps $\tau^n$ to pairwise commute.

Given a $\Lambda$-semibranching function system on $(X, \mu)$, Theorem 3.5 of \cite{FGKP} establishes that the operators $\{S_\lambda\}_{\lambda \in \Lambda} \subseteq B(L^2(X, \mu))$, given  by 
\[ S_\lambda(f)(x) = \left(\Phi_\lambda \circ \tau^{d(\lambda)}\right)^{-1/2}(x) \cdot \chi_{R_{\lambda}}(x) \cdot f(\tau^n(x)),\]
 form a representation of $C^*(\Lambda)$. For brevity, we call such a representation a {\em $\Lambda$-semibranching representation}.  Many $\Lambda$-semibranching representations are monic;  Theorem~4.5 of \cite{FGJKP-monic} establishes that the monic $\Lambda$-semibranching representations are precisely those whose  range sets $\{R_\lambda\}_{\lambda\in \Lambda}$ generate the $\sigma$-algebra of $(X, \mu)$, up to modifications by sets of measure zero.

\begin{prop}
\label{prop:permutative-are-sbfs}
Let $\Lambda$ be a finite, source-free $k$-graph and $\{ \tau_\lambda\}_{\lambda \in \Lambda}$ a permutative representation of $C^*(\Lambda)$.  Identifying $\H$ with $\ell^2(I)$ via the bijection between $I$ and $\{ e_i: i \in I\} \subseteq \H$, we can view $\{t_\lambda\}_{\lambda \in \Lambda}$ as a $\Lambda$-semibranching representation.
\end{prop}
\begin{proof}
In this setting, the domain sets $D_\lambda\subseteq I$ are given by $J_\lambda$ and the maps $\tilde \sigma_\lambda: J_\lambda \to K_\lambda$ are the prefixing maps $\tau_\lambda$.  The range sets $R_\lambda$ are therefore identified with $K_\lambda$, and the Radon--Nikodym derivatives are constantly equal to 1: 
\[ \Phi_\lambda = \frac{|K_\lambda|}{|J_\lambda| } = 1\]
since $\tilde \sigma_\lambda$ is a bijection.  
Proposition \ref{prop:encoding-map}  describes the coding maps $\tilde{\sigma}^n$ of the $\Lambda$-semibranching function system, and straightforward computations reveal that the conditions of Definition \ref{def-lambda-SBFS-1} are satisfied.
Moreover, the associated $\Lambda$-semibranching representation $\{S_\lambda\}_{\lambda\in \Lambda}$ satisfies 
\[ S_\lambda(e_i) = S_\lambda(\chi_{\{i\}}) = \begin{cases} \chi_{\{ \tilde \sigma_\lambda(i)\}} =  e_{\tilde \sigma_\lambda(i)}, & i \in J_\lambda \\
0, & \text{ else.}
\end{cases}\]
In other words, the $\Lambda$-semibranching representation agrees with the permutative representation $\{t_\lambda\}_{\lambda \in \Lambda}$, as claimed.  
\end{proof}

\begin{rmk}
It follows from the above Proposition that 
 the faithful separable representation discussed in Theorem~5.1 of \cite{FGJKP-SBFS} is also a permutative representation.
\end{rmk}

\subsection{Decomposition of permutative representations}

Theorem \ref{thm-decomp-perm-repres} below is an extension of Theorem 4.13 of \cite{dutkay-jorgensen-atomic} from the case of the Cuntz algebras $\mathcal O_N$ to the much broader setting of higher-rank graph $C^*$-algebras.  Before stating the Theorem, we recall a classical definition.
 \begin{defn}[\cite{KP}]
We say that a $k$-graph $\Lambda$ is \emph{aperiodic} if for each $v\in\Lambda^0$, there exists $x\in v\Lambda^\infty$ such that for all $m\ne n\in\N^k$ we have $\sigma^m(x)\ne \sigma^n(x)$.
 \end{defn}
\begin{thm}
\label{thm-decomp-perm-repres}
Let $\Lambda$ be a row-finite $k$-graph with no sources and let $\{t_\lambda:\lambda\in \Lambda\}$ be a purely atomic representation of $C^*(\Lambda)$ on a Hilbert space $\mathcal{H}$. If the representation
is supported on an {orbit of aperiodic paths}, then the representation is permutative. Moreover the representation can be
decomposed into a direct sum of permutative representations with injective encoding maps.
\end{thm}
\begin{proof}
Let $P$ be the projection valued measure associated to the representation $\{t_\lambda:\lambda\in \Lambda\}$  and suppose it is supported on $\Omega\subset \Lambda^\infty,$ which by our earlier results can be decomposed into orbits corresponding to a decomposition of the original representation.  So let us assume that our set $\Omega$ is equal to a single orbit of the aperiodic path $\omega\in \Omega\subset \Lambda^{\infty}.$ As in the proof of Theorem 4.13 of \cite{dutkay-jorgensen-atomic}, let $\{e_{\omega,\ell}\}_{\ell\in {\mathcal J}}$ be an orthonormal basis for $P(\{\omega\})\mathcal{H}$ for an index set $\mathcal{J}$. 

We know that every point in the orbit of $\omega$ is of the form $a\sigma^j(\omega),$ for some finite path $a$ and an element $j\in\mathbb N^k.$ Moreover, since $\omega$ is an aperiodic path, this decomposition is unique. We define an orthonormal basis on $P(\{a\sigma^j(\omega)\})\mathcal{H}$ by $\{t_at^*_{\omega(0,j)}e_{\omega,\ell}:=e_{a\sigma^j(\omega),\ell}\}_{\ell\in {\mathcal J}}.$ The results of our previous sections show that this is indeed an orthonormal basis for $P(\{a\sigma^j(\omega)\})\mathcal{H}.$ Since 
$$\Omega=\bigcup_{\gamma\in \Omega}\{\gamma\}\;=\;\bigcup_{j\in \mathbb N^k}\bigcup_{a\in \Lambda: s(a)=r(\sigma^j(\omega))}\{a\sigma^j(\omega)\},$$
we have 
$$\text{Id}_{\mathcal{H}}\;=\;P(\Omega)=\sum_{j\in \mathbb N^k} \sum_{a\in \Lambda: s(a)=r(\sigma^j(\omega))} P(\{a\sigma^j(\omega)\}),$$
where the sum converges in the strong operator topology.
Therefore, an orthonormal basis for ${\mathcal H}$ is given by 
$$\bigcup_{j\in \mathbb N^k} \bigcup_{a\in \Lambda: s(a)=r(\sigma^j(\omega))}\{t_at_{\omega(0,j)}^*e_{\omega,\ell}\}_{\ell\in {\mathcal J}}= \bigcup_{j\in \mathbb N^k} \bigcup_{a\in \Lambda: s(a)=r(\sigma^j(\omega))} \{e_{a\sigma^j(\omega),\ell}\}_{\ell\in {\mathcal J}}.$$  
One easily checks that setting
$${\mathcal H}_\ell\;=\;\overline{\text{span}}\bigcup_{j\in \mathbb N^k} \bigcup_{a\in \Lambda: s(a)=r(\sigma^j(\omega))}\{t_at_{\omega(0,j)}^*e_{\omega,\ell}=e_{a\sigma^j(\omega),\ell}\}\;=\;\{e_{\gamma,\ell}:\;\gamma\in \Omega\},$$
each ${\mathcal H}_\ell$ is an invariant subspace for the representation, and
$${\mathcal H}\;=\;\bigoplus_{\ell \in {\mathcal J}}{\mathcal H}_\ell.$$
Thus, in this case, our index set for the orthonormal basis for ${\mathcal H}$ is given by:
$$I:=\;\{(a\sigma^j(\omega)=\gamma,\ell):\; \gamma\in \Omega,\; \ell\in {\mathcal J}\}.$$

Returning to our notation in Definition \ref{defpermutative}, if $a\in \Lambda,$ 
we have  
$$J_a\;=\;\{(\gamma,\ell): \ell\in {\mathcal J},\;\gamma\in \Omega, |s(a)=r(\gamma)\}; \qquad K_a\;=\;\{(\gamma,\ell):\ell\in {\mathcal J},\;\gamma\in \Omega,\;\gamma(0,d(a))=a\}.$$
The maps 
$\tilde{\sigma}_a:J_a\to K_a$
are given by 
$$\tilde{\sigma}_a((\gamma,\ell))\;=\;(a\gamma,\ell),$$ and the coding maps are given by 
$\tilde{\sigma}^n(\gamma,\ell)\;=\;(\sigma^n(\gamma),\ell).$
The encoding map $E: I \to \Lambda^\infty$ is then given by
$$E((\gamma,\ell))\;=\;\gamma.$$

One easily calculates that the  conditions of Definition \ref{defpermutative} hold.  To complete the proof, one observes that restricting the encoding map $E$ to the basis 
\[I_\ell: = \{(\gamma, \ell):\gamma \in \Omega\}\]
for the subspace ${\mathcal H}_{\ell}$, results in an injective map. 
\end{proof}

\end{document}